\documentclass[12pt, a4paper]{amsart}

\usepackage{latexsym}
\usepackage{color}

\usepackage{amssymb}
\usepackage[centertags]{amsmath}
\usepackage{verbatim}

\usepackage{units}

\usepackage{bbm}
\usepackage{amsthm}
\usepackage{graphicx}
\usepackage{float}

\usepackage[colorlinks=true,linkcolor=red ,citecolor=magenta]{hyperref}

\renewcommand*{\eqref}[1]{%
	\hyperref[{#1}]{\textup{\tagform@{\ref*{#1}}}}%
}

\newtheorem{theorem}{Theorem}

\newtheorem{lemma}[equation]{Lemma}

\newtheorem{proposition}[equation]{Proposition}
\newtheorem{corollary}[equation]{Corollary}

\newtheorem{claim}[equation]{Claim}

\theoremstyle{definition}

\newcommand{\theoremname}{testing}
\theoremstyle{remark}
\newtheorem*{remark*}{Remark}

\numberwithin{equation}{section}

\renewcommand{\phi}{\varphi}

\newcommand{\RomanNumeralCaps}[1]
{\MakeUppercase{\romannumeral #1}}

\newcommand{\Id}{\mbox{\rm Id}}
\newcommand{\e}{\varepsilon}

\newcommand{\al}{\alpha}

\newcommand{\eps}{\varepsilon}

\newcommand{\cA}{\mathcal A}

\newcommand{\cM}{\mathcal M}

\newcommand{\cG}{\mathcal G}
\newcommand{\cO}{\mathcal O}

\newcommand{\bE}{\mathbb E}

\newcommand*{\dprime}{{\prime \prime}}

\newcommand*{\avg}[1]{\left\langle {#1} \right\rangle}
\newcommand*{\wbar}[1]{\overline{#1}}
\newcommand*{\what}[1]{\widehat{#1}}

\newcommand{\norm}[1]{\left\lVert#1\right\rVert}

\newcommand{\im}{\rm i}

\def\re#1{\mathrm{Re}\,#1}
\def\im#1{\mathrm{Im}\,#1}

\renewcommand{\le}{\leqslant}
\renewcommand{\ge}{\geqslant}

\newcommand{\bR}{\mathbb R}
\newcommand{\bC}{\mathbb C}
\newcommand{\bZ}{\mathbb Z}

\newcommand{\bT}{\mathbb T}
\newcommand{\bN}{\mathbb N}

\newcommand{\bP}{\mathbb P}

\newcommand{\ti}{\widetilde}

\title[Minimum modulus of Gaussian polynomials]{The minimum modulus of Gaussian \\ trigonometric polynomials}
\dedicatory{Dedicated to the memory of Thomas Liggett, 1944--2020}

\author{Oren Yakir}
\address{\tiny{Oren Yakir, School of Mathematics, Tel Aviv University,
	Ramat Aviv 6997801, Israel.}}

\email{oren.yakir@gmail.com}

\author{Ofer Zeitouni}
\address{\tiny{Ofer Zeitouni, Department of Mathematics, Weizmann Institute of Science, Rehovot 76100, Israel.}}

\email{ofer.zeitouni@weizmann.ac.il}

\thanks{This project has received funding from the European Research Council (ERC) under the European Union's Horizon 2020 research and innovation programme (grant agreement No. 692452). O.Y. is supported by ISF Grants 382/15 and 1903/18.}
\date{June 16, 2020.}
\begin{document}
	\maketitle
	\begin{abstract}
		We prove that the minimum of the modulus of a random trigonometric polynomial with Gaussian coefficients, properly normalized, has limiting exponential distribution. 
	\end{abstract}
	\section{Introduction}
	Let $n\ge 1$ and consider the random trigonometric polynomial given as
	\[
	P_n(x):= \frac{1}{\sqrt{2n+1}} \sum_{j=-n}^{n} \zeta_j e^{ijx}
	\] 
	where $i :=\sqrt{-1}$ and $\{\zeta_j\}$ are standard independent complex Gaussian coefficients; that is, the density of the random varible $\zeta_j$ with respect to the Lebesgue measure in the complex plane is $\frac{1}{\pi} e^{-|z|^2}$. We note that with this choice of coefficients,
	the polynomial $P=P_n$ is a mean-zero (complex-valued) stationary Gaussian process on $\bT = \bR/2\pi \bZ$ with covariance kernel given by
	\begin{equation}
	\label{eq:covariance_kernel_of_polynomial}
	r_n(x) := \bE\left[P(0)\overline{P(x)}\right] = \frac{1}{2n+1} \sum_{j=-n}^{n} e^{-ijx} = \frac{\sin\left(\left(n+\frac{1}{2}\right)x\right)}{(2n+1) \sin\left(x/2\right)}.
	\end{equation}
	In this paper we study the random variable
	\begin{equation}
	\label{eq:definition_of_minimum}
	m_n:=\min_{x\in \bT}|P_n(x)|
	\end{equation}
	and its limiting distribution as $n\to\infty$. The main result is the following.
	\begin{theorem}
		\label{thm:limiting_distribution_of_minimum}
		For all $\tau>0$ we have that
		\[
		\lim_{n\to\infty} \bP\left(m_n \ge \frac{\tau}{n}  \right) =  e^{- \lambda \tau}
		\]
		where $\lambda = 2\sqrt{\pi/3}$.
	\end{theorem} 
	\subsection{Background}
	The study of random polynomials (and in particular, their zeros) has a long history. Consider the \emph{Kac polynomial}
	\[
	F(z):= \sum_{j=0}^{n} \eta_j z^j,\quad z\in \bC,
	\]
	where $\{\eta_j\}$ is an i.i.d. sequence of complex random variables. It is well known that if $\bE\log\left(1+ |\eta_0|\right)<\infty$ then
	  the zeros of $F$ concentrate uniformly around the unit
	  circle as the degree $n$ tends to infinity 
	  \cite{erdos_turan,sparo_sur} (see \cite{hughes_nikeghbali}
	  for a more modern perspective).  For finer
	results  and additional references, see
	\cite{shepp_vanderbei} for the Gaussian coefficients case (i.e. when $\eta_j = \zeta_j$) or	\cite{ibragimov_zeitouni} for the general case.
	
	In view of these results, it is natural to expect that
	the random variable 
	$$m_n(F) := \min_{ |z| = 1} |F(z)| $$ 
	tends to zero as $n\to\infty$, and to study the order of magnitude at which this random variable decay. A particular case of this problem, when the coefficients $\eta_j$ are Rademacher random variables (that is, $\eta_j$ takes the values $\{\pm 1\}$ with equal probability), was posed already by Littlewood in \cite{littlewood}. In \cite{konyagin}, Konyagin proved that in the Rademacher case, for all $\eps>0$
	\[
	\bP\left( m_n(F) \ge \frac{1}{n^{1/2-\eps}} \right) \xrightarrow{n\to\infty } 0.
	\]
	In a later paper, Konyagin and Schlag \cite{konyagin_schlag} proved
	that for either  the Rademacher or  Gaussian cases, there exists some absolute constant $C>0$ such that
	\[
	\limsup_{n\to\infty} \bP\left(m_n(F) < \frac{\eps}{\sqrt{n}}\right) \leq C\eps 
	\] 
	for all $\eps>0$. 
	Note that for the case of complex Gaussian coefficients, $m_n(F)$ is exactly $m_n$ (up to a normalization by $1/\sqrt{n}$) as defined in (\ref{eq:definition_of_minimum}), so Theorem \ref{thm:limiting_distribution_of_minimum} resolves this question for the Gaussian case. The same method of proof works
(after some minor modifications) in the case of real Gaussian coefficients, see Section \ref{sec:real_gaussian} for more details.	
	
After the completion of this paper, we learnt \cite{cook_nguyen} 
that N. Cook and H. Nguyen proved a universality result 
for the minimum modulus of random polynomials  with i.i.d. coefficients,
by a comparison method. Their starting point is our Theorem  
\ref{thm:limiting_distribution_of_minimum} for the Gaussian case.
In particular, their work settles  the 
problem for the Rademacher coefficients case.  
	\subsection{Structure of the proof}
		The proof of Theorem
		  \ref{thm:limiting_distribution_of_minimum} is based on the
		 observation  that locally (within intervals of length
		much smaller than $1/n$), the polynomial $P_n$ is well
		approximated by its linear interpolation. This observation 
		is a consequence of a-priori bounds on the second 
		derivative, see Lemma \ref{lemma:second_derivative_is_small}.
		In particular, by
		the ``high-school" exercise in Section \ref{subsec-highschool},
		the value and the location of local minima of $P_n$ can be well
		predicted by linear interpolation from an appropriate
		net of points. Crucially,
		this observation also implies that points which
		are candidates for being global minima are well separated, 
		see
		Lemma \ref{lemma:points_in_the_process_are_separated}.
		
		Introduce a net of points $x_\alpha\in \bT$, and set 
		$X_\alpha$ to be
		a signed version
		of $n|P_n(x_\alpha^*)|$,
		where
		$x_\alpha^*$ is the location of the minimum of 
		$P_n(\cdot)$ based
		on linear interpolation from $(P_n(x_\alpha), P_n'(x_\alpha))$.
		Introduce a
	 ``good'' event	$\cA_{\al}$  that is typical for 
		global minima, see (\ref{eq-defA}) for the precise definition.
		The global minimum  $nP_n(\cdot)$ is then well approximated
		by the point closest to $0$ of the
		point process
		$\mathcal{M}_n := \sum_{\al} \delta_{X_\al}
		\mathbbm{1}_{\cA_\al} $.
		Theorem  \ref{thm:limiting_distribution_of_minimum}
		is then a consequence of the fact that $\mathcal{M}_n$ converges to a Poisson point process of intensity $\sqrt{\pi/3}$ 
		(the intensity is computed in Corollary \ref{cor:limit_intensity_of_point_process}). 
		The Poisson convergence, in turn, is based on
		a characterization of Poisson processes due to Liggett
	      \cite{liggett}, and uses
		a technique introduced by Biskup and 
		Louidor in \cite{biskup_louidor}:
		one
		exploits the fact that $P_n$ is a Gaussian process
		and that minima are well separated to deduce 
		an invariance property of $\mathcal{M}_\infty$ with respect
		to additive i.i.d. perturbations of the points $X_\alpha$.
		The details of this argument appear in Section 
		\ref{sec-liggett}.
		
		We remark that most of the above argument does 
		not use the Gaussian nature of the coefficients 
		in any essential way. 
		The only place in the argument where the coefficients of 
		$P_n$ are required to be Gaussian 
		(or rather, have a `not too small' Gaussian component) is in 
		Section \ref{sec-liggett}, where we extract a Poisson limit 
		from Liggett's characterization. 
		In Section \ref{sec-small_gaussian_component} we give a sketch
		of how our result can be generalized to 
		a more general choice of random coefficients, 
		having a small Gaussian component. Of course, this extension
		is covered by the Cook-Nguyen 
		theorem mentioned above, and the sketch just serves to 
		illustrate the flexibility, together with the limitations, of our approach.

	\subsection{A high-school exercise}
\label{subsec-highschool}
	Suppose we are given two (non-zero) planar vectors $A= (a_1,a_2)$ and $B=(b_1,b_2)$. We want to find the distance between the origin and the straight line $\left\{A+tB\mid t\in \bR \right\}$. Set $F(t):= A + tB$ and let $t_{min}$ be defined via the relation
	\[
	|F(t_{min})| = \min_{t\in \bR} |F(t)|.
	\]
	We denote by $\gamma$ the angle between $A$ and $B$. It is evident (see Figure \ref{fig:minimum_of_linear_function}) that
	\[
	\frac{\avg{A,B}}{|B|}  = |A| \cos(\gamma)  = |F(t_{min}) - A| = -t_{min}|B|
	\]
	and so $t_{min} = -\avg{A,B}/|B|^2$. Now, simple algebra yields that $$F(t_{min}) = \frac{a_1b_2 - a_2 b_1}{\sqrt{b_1^2 + b_2^2}} = \frac{\avg{A,B^\perp}}{|B|} $$ 
	where $B^\perp$ is an anti-clockwise rotation of the vector $B$ by $90^\circ$ (see again Figure \ref{fig:minimum_of_linear_function}). In complex notation, by considering $A = a_1 + i a_2$ and $B=b_1 + i b_2$, we have
	\[
	t_{\min} = -\frac{\re{(A\overline{B})}}{|B|^2}, \quad F(t_{\min}) = \frac{\im{(A\overline{B})}}{|B|}. 
	\]
\begin{figure}[H]
	\begin{center}	
		\scalebox{0.25}{\includegraphics{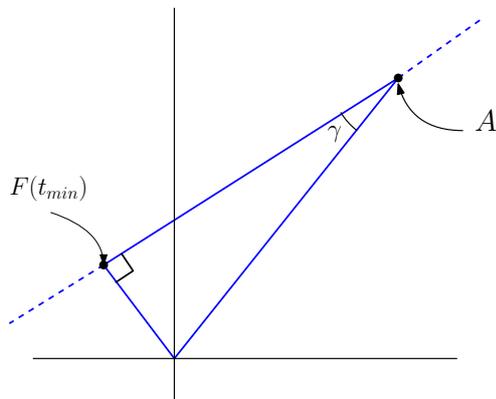}}
		\caption{The dashed line is $\left\{A+tB \mid t\in \bR\right\}$. We see that $\cos(\gamma)|A| = |A-F(t_{min})|$ and that $F(t_{min})$ is the projection of the vector $A$ onto the straight line perpendicular to $B$.}
		\label{fig:minimum_of_linear_function}
	\end{center}
\end{figure}

	\subsection*{Notation}
	We write $f\ll g$ or $f= \cO(g)$ if there exist a constant $C>0$ that does not depend on $n$ such that $f\leq C g$. We will also write $f = o(g)$ if $f/g\to 0$ as $n\to\infty$. We denote by $dm(\cdot)$ the Lebesgue measure on $\bC$, and by
	$C_c(\bR)$  the space of continuous, compactly supported functions
      on $\bR$. We write
    $\mathcal{N}_{\bR}(a,b)$ for the Gaussian law with mean $a$ and variance
    $b$. For random variables $X$ and $Y$, we
    write $X\stackrel{\mbox{\scriptsize{\rm law}}}{=} Y$ if they are identically distributed. For a sequence of random variables $X_n$, we write $X_n\xrightarrow{\ d \ } X$ if
  $X_n$ converges in distribution to $X$ as $n\to\infty$. Finally, for 
  $N\in \bN$ even we write $[N] := \left\{-N/2,-N/2+1,\ldots,N/2-1\right\}$.
  
	\section{Point process of near-minima values}
	\label{sec-point_process_of_near_minima}
	Fix some $\e>0$ small (that will not depend on $n$; $\e=1/100$ is good enough) and set $N:=2\lfloor n^{2-\e} /2\rfloor$ so that
	$N$ is even. We consider $N$ equidistributed points on the unit circle given by
	\[
	x_\al  = \frac{2\pi \al}{N}, \quad \al = -N/2,\ldots,N/2-1.
	\] 
	Denote the interval of length $2\pi/N$ centered at the point $x_\al$ by $I_\al$, namely,
	\[
	\bT = \bigcup_{\al = -N/2}^{N/2-1} I_\al, \quad I_\al = \left[x_\al - \frac{\pi}{N}, x_\al + \frac{\pi}{N}\right].
	\]
	The linear approximation for the polynomial at the point $x_\al$ is
	\begin{equation}
	  \label{eq-Falpha}
      F_\al(x) := P(x_\al) + (x-x_\al) P^\prime(x_\al).
    \end{equation}
    Following the high-school exercise from Section \ref{subsec-highschool}, we set
	\begin{equation*}
		Y_\al := -\frac{\re\left(P(x_\al)\overline{P^\prime(x_\al)}\right)}{|P^\prime(x_\al)|^2}, \qquad Z_\al := n\cdot \frac{\im\left(P(x_\al)\overline{P^\prime(x_\al)}\right)}{|P^\prime(x_\al)|}.
	\end{equation*}
	And so, $Z_\al$ is the minimal modulus (kept with a sign and scaled by $n$) of the linear approximation $F_\al$ and $Y_\al$ is the unique point such that $|F_\al(Y_\al)| = |Z_\al|/n$. The event that the interval $I_\al$ produce a candidate for the minimal value is given by
	\begin{align}
	  \label{eq-defA}
		\cA_\al = &\cA_\al^\prime \cap \cA_\al^\dprime ,\qquad
	\mbox{\rm where} \nonumber \\
		&\cA_\al^\prime = \left\{Y_\al \in I_\al, \ |Z_\al| \leq \log n \right\}  \\ &\cA_\al^\dprime =\left\{|P(x_\al)| \leq n^{-1/2}, \ |P^\prime(x_\al)| \in \left[n^{1-\e/2}, C_0 n\sqrt{\log n} \right]  \right\}.\nonumber
	      \end{align}
	      $C_0$ in the definition above is a large absolute constant which we specify in Lemma \ref{lemma:probability_of_single_point_to_be_near_minima}; $C_0=10$ is good enough. The event $\cA_\al^\prime$ tells us that the interval $I_\al$ gives a candidate for the minimum and the event $\cA_\al^\dprime$ is just the typical values of $\left(P(x_\al),P^\prime(x_\al)\right)$ so that the interval $I_\al$ gives a candidate. We can now define the point process on $\bR$ of near-minima values as
	\begin{equation}
		\label{eq:definition_of_extremal_process}
		\mathcal{M}_n := \sum_{\al=-N/2}^{N/2-1} \delta_{X_\al}, \quad \text{where } \  X_\al = Z_\al \cdot \mathbbm{1}_{\cA_\al} + \infty \cdot {\mathbbm{1}}_{\cA_\al^c}.
	\end{equation} 
	 Here and throughout, we consider $\mathcal{M}_n$ as an element
	 of the space of locally finite, integer valued positive
	 measures on $\bR$, equipped with the local weak$^*$ 
       topology generated by bounded, compactly supported functions. Thus, we never consider the points at infinity that are contributed
     by the events $\cA_\al^c$.
     
     Essentially, the linear approximations captures the global minimum of the polynomial since the second derivative is small. In what follows we make this idea precise. For $\beta>0$, define the event
	 \[
	   \mathcal{G}_{\beta} = \left\{ \sup_{x\in \bT} |P^\dprime_n(x)| \leq n^{2+\beta}\right\}.
	 \]
	 \begin{lemma}
	 	\label{lemma:second_derivative_is_small}
		For any $\beta>0$ we have $\bP(\mathcal{G}_{\beta}^c) \ll \exp(-n^\beta)$.
	 \end{lemma}
	 \begin{proof}
	 	Denote by $g(x) := n^{-2} \re{P^\dprime(x)}$ just for the proof. Then $g$ is a trigonometric polynomial such that for all $x\in \bT$,
		  $g(x)$ is distributed $\mathcal{N}_{\bR} (0,\sigma^2)$ where $\sigma^2 = \frac{1}{10} + o(1)$. Hence, for all $\theta\in \bR$ we have that
	 	\begin{equation}
	 	\label{eq:moment_generating_function_of_second_derivative}
	 	\bE\left[e^{\theta g(x)}\right] = e^{\theta^2/(5+o(1))}.
	 	\end{equation}
	 	Let $\ti{x}$ be the point such that $|g(\ti{x})| = \norm{g}_\infty$, and recall that $\norm{g^\prime}_\infty \leq (2n+1) \norm{g}_\infty$ by Bernstein's inequality. Provided 
		that $|x-\ti{x}|\leq 1/4n$, we have that
	 	\begin{align*}
	 	|g(x)| &\ge |g(\ti{x})| - |g(\ti{x}) - g(x)| \ge \norm{g}_\infty - |\ti{x}-x| \norm{g^\prime}_\infty \geq \frac{1}{2} \norm{g}_\infty.
	 	\end{align*}
		Combining this with (\ref{eq:moment_generating_function_of_second_derivative}) and Fubini, we get the bound 
	 	\begin{align*}
		  \bE\left[e^{\theta \norm{g}_\infty/2}\right] &\leq 2n\bE\left[\int_{|x-\ti{x}| \le 1/4n} \left( e^{\theta g(x)} + e^{-\theta g(x)} \right) dx\right] \\ &\leq 2n \bE\left[\int_{\bT} \big(e^{\theta g(x)} + e^{-\theta g(x)}\big) dx \right] \leq 8\pi n e^{\theta^2/5}.
	 	\end{align*}
	 	Now, we can use the 
		Markov inequality with $\theta = 2\sqrt{2}$ and see that
	 	\begin{align*}
	 	\bP\left(\mathcal{G}_{\beta}^c\right) \leq 2\bP\left(\norm{g}_\infty \ge \frac{n^\beta}{\sqrt{2}}\right) \leq \bE\left[e^{\theta\norm{g}_\infty} \right]e^{-\theta n^{\beta}/\sqrt{2} } \ll e^{-n^\beta}.
	 	\end{align*}
	 \end{proof}
 	We now turn to compute the probability that the interval $I_\al$ contributed a point to $\cM_n$. 
 	\begin{lemma}
 		\label{lemma:probability_of_single_point_to_be_near_minima}
		For any interval $[a,b]\subset \bR$ and for all $\al \in [N]$ we have
 		\begin{equation*}
 			\bP\left(X_\al \in [a,b] \right) = \sqrt{\frac{\pi}{3}}\cdot \frac{b-a}{N} + o\left(\frac{1}{N}\right),
 		\end{equation*} 
 		as $n\to\infty$.
 	\end{lemma}
 	\begin{proof}
 		The proof is a simple Gaussian computation. By stationarity, we may assume that $x_\al = 0$. Set $\sigma_n:=\sqrt{\frac{n(n+1)}{3}}$, so that
		$\left(P(0),P^\prime(0)/\sigma_n\right)$ are independent standard complex Gaussian random variables. Indeed, we note that
		\begin{align*}
			\bE\left[P(0)\overline{P^\prime(0)}\right] &= \frac{1}{2n+1}\sum_{j=-n}^{n} -ij = 0, \quad \text{and that } \\ 
			\bE\left[|P^\prime(0)|^2\right]& = \frac{1}{2n+1} \sum_{j=-n}^{n} j^2 = \frac{n(n+1)}{3}.
		\end{align*}
		Moreover, a straight forward computation shows that $\bE\left[P(0)^2\right] = \bE\left[P(0)P^\prime(0)\right] = \bE\left[P^\prime(0)^2\right] = 0$. By the bounded convergence theorem
 		\begin{align*}
		  \bP\big(\cA_\al^\prime &\cap \left\{Z_\al \in [a,b] \right\}\big) \\ &= \bP\left(\left|Y_\al\right| \leq \frac{\pi}{N}, Z_\al \in [a,b] \right) \\ &= \int_{\bC} e^{-|w|^2} \bP\left(\frac{\left|\re\left(P(0)\wbar{w}\right)\right|}{|w|^2} \leq \frac{\pi \sigma_n}{N} , \ n\cdot \frac{\im \left(P(0)\wbar{w}\right)}{|w|} \in \left[a,b\right] \right)\frac{dm(w)}{\pi} \\ &= \int_{\bC} e^{-|w|^2} \bP\left(\left|\re\left(P(0)\right)\right| \leq \frac{\pi \sigma_n}{N}|w|\right)\cdot \bP\Big(\im\left(P(0)\right) \in \left[\frac{a}{n},\frac{b}{n}\right]\Big) \frac{dm(w)}{\pi} \\ &= \frac{2\pi \sigma_n}{\pi^2 N} \cdot \frac{b-a}{n} \int_{\bC} e^{-|w|^2}|w| dm(w) + o(N^{-1}) \\ &= \sqrt{\frac{\pi}{3}} \cdot \frac{b-a}{N} + o(N^{-1}),
 		\end{align*}
 		where in the third equality we used the rotational symmetry of the Gaussian distribution. To conclude the lemma, we show that $\bP\left(\cA_{\al}^\prime \cap (\cA_\al^\dprime)^c  \cap \left\{Z_\al \in [a,b]\right\} \right)$ is negligible. Since $|Z_\al|/n = |P(0) + Y_\al P^\prime(0)|$,
		the triangle inequality yields that on the event $\cA_\al^\dprime \cap \{ Z_\al \in [a,b]\},$
 		\begin{equation}
 			\label{eq:bound_on_value_of_poly_assuming_crossing}
 			|P(0)| \leq |Z_\al|/n  + |Y_\al||P^\prime(0)| \ll \frac{1}{n} + C_0\frac{n\sqrt{\log n}}{n^{2-\e}}  \ll \frac{1}{n^{1-2\e}}.
 		\end{equation}
 		From (\ref{eq:bound_on_value_of_poly_assuming_crossing}) we conclude that
 		\[
 		\left\{|P(0)| \ge n^{-1/2},\ Z_\al \in [a,b]\right\} \subset \left\{ |P^\prime(0)|\ge C_0 n\sqrt{\log n} ,\ Z_\al \in [a,b]\right\}
 		\]
 		for all fixed $C_0>0$ and for $n$ large enough. 
		We thus get the upper bound
 		\begin{align*}
 			\bP&\left(\cA_{\al}^\prime \cap (\cA_\al^\dprime)^c \cap \left\{Z_\al \in [a,b]\right\} \right) \\ & \leq \bP\left(\cA_{\al}^\prime \cap \left\{Z_\al \in [a,b]\right\} \cap \left\{|P^\prime(0)| \leq n^{1-\e/2} \right\} \right) + \bP\left(|P^\prime(0)| \ge C_0 n\sqrt{\log n}\right) \\
			&= \cO\left(\frac{1}{n^2}\right).
 		\end{align*}
 		The first probability is $\cO(n^{-2})$ from the 
		same Gaussian computation as done above and 
		the second probability is $o(n^{-2})$ for a large absolute constant $C_0$ because $P'(0)$ is a complex Gaussian variable
		of variance bounded by $n^2$; one sees that $C_0=10$ will do.
		Altogether,
 		\begin{align*}
 			\bP\left(X_\al \in [a,b]\right) &= 
			\bP\left(\left\{Z_\al \in [a,b] \right\} \cap \cA_\al\right)\\
			&=
			\bP\left(\left\{Z_\al \in [a,b] \right\} \cap \cA_\al^\prime \right)
		      - \bP\left(\left\{Z_\al \in [a,b] \right\} \cap \cA_\al^\prime \cap (\cA_\al^\dprime)^c\right) \\ &= \sqrt{\frac{\pi}{3}} \cdot \frac{b-a}{N} + o\left(\frac{1}{N}\right).  
 		\end{align*} 		
 		\end{proof}
 	\begin{corollary}
 		\label{cor:limit_intensity_of_point_process}
		For any interval $[a,b]\subset \bR$ we have, as $n\to\infty$,
 		\[
 		\bE\left[\cM_n([a,b]) \right] = \sqrt{\frac{\pi}{3}} (b-a) + o(1).
 		\]
 	\end{corollary}
 	\begin{corollary}
 		\label{cor:tightness_of_point_process}
		The sequence of point processes $\{ \cM_n \}$ is tight. That is, for any interval $I\subset \bR$ and for all $n\ge n_0$,
 		\[
 		\lim_{K\to\infty} \bP\left(\cM_n(I) \ge K \right) = 0.
 		\]
 	\end{corollary}
 	Both Corollary \ref{cor:limit_intensity_of_point_process} and Corollary \ref{cor:tightness_of_point_process} are immediate consequences of Lemma \ref{lemma:probability_of_single_point_to_be_near_minima}. We turn to prove that the extremal process $\cM_n$ captures the minimum modulus of our polynomial $P_n$.
 	\begin{lemma}
 		\label{lemma:extremal_process_captures_minimum}
 		For all $\tau>0$ we have that
 		\[
 		\lim_{n\to\infty} \left|\bP\left(m_n \ge \frac{\tau}{n}  \right) - \bP\left(\cM_n((-\tau,\tau)) = 0 \right)\right| =0.
 		\]
 	\end{lemma}
 	\begin{proof}
 		Clearly,
 		\begin{equation*}
 		\bP\left(m_n \ge \frac{\tau}{n}\right) = \bP\left(\bigcap_{\al} \left\{\min_{x\in I_\al}|P_n(x)| \ge \frac{\tau}{n} \right\}\right).
 		\end{equation*}
		Recall the definition of the linear approximation (\ref{eq-Falpha}). For each $\al$ we use Taylor expansion and see that on the event $\cG_\e$,
 		\[
 		\left|P_n(x) - F_\al(x) \right| \ll \frac{\sup_{x\in \bT}|P^\dprime(x)|}{N^2} \ll \frac{1}{n^{2-3\e}}.
 		\]
		Hence, for large enough $n$, we have the upper bound
 		\begin{align*}
		  &\left|\bP\left(m_n \ge \frac{\tau}{n}\right) - \bP(\cM_n((-\tau,\tau)) = 0)\right|  \\ & \leq \sum_{\al=-N/2}^{N/2-1} \bP\left(\left\{\min_{x\in I_\al}|P_n(x)| \ge \frac{\tau}{n},\ |X_\al| < \tau \right\}\cap \cG_\e 
		  \right) \\ & \qquad + \sum_{\al=-N/2}^{N/2-1} \bP\left(\left\{\min_{x\in I_\al}|P_n(x)| < \frac{\tau}{n},\ |X_\al| \geq \tau \right\}\cap 
		  \cG_\e \right) + 2\bP\left(\mathcal{G}^c_\e\right) \\ &\leq \sum_{\al=-N/2}^{N/2-1} \bP\left( \tau - \frac{1}{2n^{1-3\e}}<|X_\al| \leq \tau\right) + \bP\left( \tau <|X_\al| \leq \tau + \frac{1}{2n^{1-3\e}}\right) + \cO\left(e^{-n^\e}\right). 
 		\end{align*}
 		Following the same computation we did in the proof of Lemma \ref{lemma:probability_of_single_point_to_be_near_minima}, we see that
 		\[
 		\bP\left( \tau - \frac{1}{n^{1-3\e}}<|X_\al| \leq \tau\right) = 2\sqrt{\frac{\pi}{3}}\frac{\tau}{N n^{1-3\e}}(1+o(1))
 		\]
 		and a similar equality for the other term in the sum. Altogether, we see that
 		\[
 		\left|\bP\left(m_n \ge \frac{\tau}{n}\right) - \bP(\cM_n((-\tau,\tau)) = 0)\right| \ll \frac{1}{n^{1-3\e}}
 		\]
 		and we are done.
 	\end{proof}
 	By Lemma \ref{lemma:extremal_process_captures_minimum}, the limit distribution of $m_n$ (Theorem \ref{thm:limiting_distribution_of_minimum}) will follow if we can show that
	$\cM_n$ converge in distribution (as a point process) to a Poisson point process with the desired intensity. This is established in what follows. We first prove that the points in the extremal process come from well separated intervals.
 	\begin{lemma}
 		\label{lemma:points_in_the_process_are_separated}
 		For all $\e>0$ we have that
 		\begin{equation*}
 			\lim_{n\to\infty} \bP\bigg(\bigcup_{\stackrel{\al\not=\al^\prime}{|x_\al - x_\al^\prime|\leq n^{-\e}}} \cA_\al \cap \cA_{\al^\prime} \bigg) = 0.
 		\end{equation*}
 	\end{lemma}
 	To prove Lemma \ref{lemma:points_in_the_process_are_separated} we will need two claims. We first prove the lemma assuming both claims, and then turn to prove each claim separately. Denote by
 	\begin{equation}
 	\label{eq:two_sums_in_separation_lemma}
 	S_\RomanNumeralCaps{1} := N \sum_{\al=1}^{\lfloor n^{1-2\e}/2\pi \rfloor} \bP\left(\cA_0 \cap\cA_\al \right), \qquad S_\RomanNumeralCaps{2} := N \sum_{\al=\lceil n^{1-2\e}/2\pi \rceil}^{\lfloor n^{2-2\e}/2\pi \rfloor} \bP\left(\cA_0 \cap\cA_\al \right).
 	\end{equation}
 	\begin{claim}
 		\label{claim:first_sum_in_seperatin_lemma_is_small}
 		$S_\RomanNumeralCaps{1} = o(1)$ as $n\to\infty$.
 	\end{claim}
 	
 	\begin{claim}
 		\label{claim:second_sum_in_seperatin_lemma_is_small}
 		$S_\RomanNumeralCaps{2} = o(1)$ as $n\to\infty$.
 	\end{claim}
 	\begin{proof}[Proof of Lemma \ref{lemma:points_in_the_process_are_separated}]
 		Applying the union bound and stationarity we see that
 		\begin{align}
 		\bP\bigg(\bigcup_{\stackrel{\al\not=\al^\prime}{|x_\al - x_\al^\prime|\leq n^{-\e}}} \cA_\al \cap \cA_{\al^\prime} \bigg) & \leq \sum_{\stackrel{\al\not=\al^\prime}{|x_\al - x_\al^\prime|\leq n^{-\e}}} \bP\left(\cA_\al \cap \cA_{\al^\prime} \right) \nonumber \\ \nonumber &\ll N\sum_{\al=1}^{\lfloor n^{2-2\e}/2\pi \rfloor} \bP\left(\cA_0 \cap\cA_\al \right) = S_\RomanNumeralCaps{1} + S_\RomanNumeralCaps{2} = o(1),
 		\end{align}
 		where
		the last equality is due to Claims \ref{claim:first_sum_in_seperatin_lemma_is_small} and \ref{claim:second_sum_in_seperatin_lemma_is_small}.
 	\end{proof}
 	\begin{proof}[Proof of Claim \ref{claim:first_sum_in_seperatin_lemma_is_small}]
 		Fix some $\beta<\epsilon/2$, and consider first the term $\al = 1$ in the sum $S_\RomanNumeralCaps{1}$. Observe that for all $x\in I_1$,
 		\begin{equation*}
 		\left|F_0(x) - F_1(x)\right| \leq \left|F_0(x) - P(x)\right| + \left|F_1(x) - P(x)\right| \ll \frac{\sup_{x\in \bT}|P^\dprime(x)|}{N^2},
 		\end{equation*}
 		which yields that on the event $\cG_\beta$, for all $x\in I_1$ we have that
 		\begin{equation}
 			\label{eq:difference_between_two_adjacent_linear_interpolations_is_small}
 			|F_0(x) - F_1(x)| \ll n^{-2+3\e}.
 		\end{equation}
 		Following the same computation as done in the proof of Lemma \ref{lemma:probability_of_single_point_to_be_near_minima}, we see that
 		\[
 		\bP\left(\cA_0 , \ |Y_0| \in \left[\frac{\pi}{N} - \frac{\pi}{Nn^{\e/4}},\frac{\pi}{N}\right]\right) \ll \frac{1}{Nn^{\e/4}},
 		\]
 		whence, on the event $\cA_0 \cap \left\{|Y_0| \leq \frac{\pi}{N}-\frac{\pi}{Nn^{\e/4}}\right\}$ we have that
 		\begin{equation}
 			\label{eq:lower_bound_for_linear_interpolation_end_of_interval}
 			\left|F_0\left(\frac{\pi}{N}\right)\right| = \left|\frac{Z_0}{n} + \left(\frac{\pi}{N}-Y_0\right)P^\prime(0)\right| \gg \frac{n^{1-\e/2}}{Nn^{\e/4}} - \frac{\log n}{n} \gg \frac{1}{n^{1-\e/4}}.
 		\end{equation}
 		Therefore, we combine (\ref{eq:difference_between_two_adjacent_linear_interpolations_is_small}) and (\ref{eq:lower_bound_for_linear_interpolation_end_of_interval}) to conclude that on the event $\cG_\beta \cap \cA_0 \cap \left\{|Y_0| \leq \frac{\pi}{N}-\frac{\pi}{Nn^{\e/4}}\right\}$ we have that for all $x\in I_1$, 
 		\[
 		|F_1(x)| \geq |F_0(x)| - |F_1(x) - F_0(x)| \geq \left|F_0\left(\frac{\pi}{N}\right)\right| - |F_1(x) - F_0(x)| \gg \frac{1}{n^{1-\e/4}},   
 		\]
 		which implies that the event $\cA_1$ does not hold. 
		Altogether, we apply Lemma \ref{lemma:second_derivative_is_small} and obtain that
 		\begin{align*}
 		N\bP\left(\cA_0 \cap \cA_1 \right) &\leq N\bP\left(\cA_0 \cap \left\{ |Y_0| \in \left[\frac{\pi}{N} - \frac{\pi}{Nn^{\e/4}},\frac{\pi}{N}\right] \right\}\right) \\ & \qquad  + N\bP\left( \cA_0 \cap \cA_1 \cap \left\{|Y_0| \leq \frac{\pi}{N}-\frac{\pi}{Nn^{\e/4}} \right\}\right) \\
		&\ll n^{-\e/4} + N\bP\left(\cG^c_\beta\right) \ll n^{-\e/4}.
 		\end{align*}
 		The treatment of the rest of the sum $S_\RomanNumeralCaps{1}$ is similar, only that we do not have to impose the extra separation within the interval. We have for all $x\in I_\al$ that
 		\begin{equation}
 		\label{eq:difference_between_linear_approximations_separation_lemma}
 		\left|F_0(x) - F_\al(x)\right| \le \left|F_0(x) - P(x)\right| + \left|F_\al(x)  - P(x)\right| \ll |x_\al|^2 \sup_{x\in \bT}|P^\dprime(x)|.
 		\end{equation}
 		Furthermore, on the event $\cA_0$, we have the lower bound 
 		\begin{equation}
 		\label{eq:lower_bound_for_linear_approximation_separation_lemma}
 		|F_0(x)| \ge \left|F_0\left(x_\al - \frac{\pi}{N}\right)\right| \geq \left|x_{\al-1}\right||P^\prime(0)| - \frac{\log n}{n} \gg |x_{\al-1}| n^{1-\e/2}
 		\end{equation}
 		for all $x\in I_\al$. Recalling that
 		$\beta< \e/2$, we combine (\ref{eq:difference_between_linear_approximations_separation_lemma}) and (\ref{eq:lower_bound_for_linear_approximation_separation_lemma}) to see that on the event $\cG_\beta$,
 		\[
 		\forall x\in I_\al, \qquad |F_\al(x)| \gg \left|x_{\al-1}\right|n^{1-\e/2} + |x_\al|^2 n^{2+\beta} \gg \frac{1}{n^{1-\e/2}}
 		\]
 		(here we use that $n^{-1-\e} \gg |x_{\al-1}| \ge 2\pi/N$) which implies that $\cA_\al$ does not hold. We thus have,
 		\[
 		S_{\RomanNumeralCaps{1}} \ll n^{-\e/4} + N
 		\cdot \lfloor n^{1-2\e}/2\pi \rfloor\cdot \bP\left(\cG^c_\beta\right) \ll n^{-\e/4} +  n^{3-3\e} e^{-n^{\beta}} = o(1).
 		\]
 	\end{proof}
 	\begin{proof}[Proof of Claim \ref{claim:second_sum_in_seperatin_lemma_is_small}]
 		Denote by $\Id$ the identity matrix. The random vector
 		$
 		V_\al:=\left(P(0),\frac{P^\prime(0)}{\sigma_n}, P(x_\al),\frac{P^\prime(x_\al)}{\sigma_n} \right)
 		$ is a mean zero complex Gaussian with independent real 
		and imaginary component; the covariance matrix of
		both the real and imaginary parts is
		$\frac{1}{2}\begin{pmatrix}
 		\Id & \Sigma_\al \\ \Sigma_\al & \Id
 		\end{pmatrix},$ where
 		\[
 		\Sigma_\al := \begin{pmatrix}
 		r_n(x_\al) & -r_n^\prime(x_\al)/\sigma_n \\  r_n^\prime(x_\al)/\sigma_n & -r_n^\dprime(x_\al)/\sigma_n^2
 		\end{pmatrix},
 		\]
 		and $r_n$ is given by 
		(\ref{eq:covariance_kernel_of_polynomial}).
		The density of the vector $V_\al$ in $\bC^4$ (or $\bR^8$) is bounded from above by a constant multiple of
 		\begin{align}
 		\label{eq:determinant_computation_for_density_V}
 		\left[\det\begin{pmatrix}
 		\Id & \Sigma_\al \\ \Sigma_\al & \Id
 		\end{pmatrix}\right] ^{-4} &= \left[\det\left(\Id - \Sigma_\al^2\right) \right]^{-4} \\ &= \left[\det\left(\Id - \Sigma_\al\right) \det\left(\Id + \Sigma_\al\right) \right]^{-4}. \nonumber
 		\end{align}
 		Differentiating (\ref{eq:covariance_kernel_of_polynomial}) twice and applying
		some trigonometric identities yield that
 		\begin{align*}
 		&r_n(x) = \frac{\sin\left(\left(n+\frac{1}{2}\right)x\right)}{(2n+1)\sin(x/2)}, \\ &r_n^\prime (x) = \frac{1}{2\sin(x/2)} \left[\cos\left(\left(n+1/2\right)x\right) - \frac{\cos(x/2) \sin\left(\left(n+\frac{1}{2}\right)x\right)}{(2n+1)\sin(x/2)}\right],  \\ & r_n^\dprime (x) = \frac{n^2 \sin \left(\left(n+\frac{3}{2}\right) x\right) + (n+1)^2 \sin \left(\left(n-\frac{1}{2}\right) x\right) - (2 n (n+1)-1) \sin \left(\left(n+\frac{1}{2}\right) x\right)}{\sin^3(x/2)( 8 n+4)}.
 		\end{align*}
 		Using Taylor's approximation and some algebra it is evident that for $n^{-1-2\e} \leq |x_\al| \leq \e/n$ we have
 		\begin{align}
 		\label{eq:taylor_expansion_for_correlations}
 		r_n(x_\al) &= 1- \frac{(n x_\al) ^2}{6} + \frac{(nx_\al)^4}{120} + \cO\left(|nx_\al|^6\right), \nonumber \\ \frac{r_n^\prime(x_\al)}{\sigma_n} &= - \frac{nx_\al}{\sqrt{3}} + \frac{(nx_\al)^3}{10\sqrt{3}} + \cO\left(|nx_\al|^5\right), \\ -\frac{r_n^\dprime(x_\al)}{\sigma_n^2} &= 1 - \frac{(nx_\al)^2}{10} + \cO\left(|nx_\al|^4\right). \nonumber
 		\end{align}
 		Using (\ref{eq:taylor_expansion_for_correlations}), we see that
 		\begin{align*}
 		\det\left(\Id - \Sigma_\al \right) &= \big(1- r_n(x_\al)\big)\left(1 + r_n^\dprime(x_\al)/\sigma_n^2\right) + \big(r_n^\prime(x_\al)/\sigma_n \big)^2 = \frac{(n x_\al)^2}{3} + \cO(|n x_\al|^4), \\ \det\left(\Id + \Sigma_\al \right) &= \big(1 + r_n(x_\al)\big)\left(1 - r_n^\dprime(x_\al)/\sigma_n^2\right) + \big(r_n^\prime(x_\al)/\sigma_n \big)^2 = 4 + \cO(|n x_\al|^2),
 		\end{align*}
 		which together with (\ref{eq:determinant_computation_for_density_V}) implies that
 		\begin{equation}
 		\label{eq:bound_on_det_density_of_V}
 		\left[\det\begin{pmatrix}
 		\Id & \Sigma_\al \\ \Sigma_\al & \Id
 		\end{pmatrix}\right] ^{-4} \ll  \frac{1}{(nx_\al)^8} \qquad \text{for } \  n^{-1-2\e} \leq |x_\al| \leq \e/n.
 		\end{equation} 
 		Furthermore, for $|x_\al| \ge \e/n$, the density of $V_\al$ is uniformly bounded from above by a constant $C=C_\e$. Combining this observation with (\ref{eq:bound_on_det_density_of_V}), we can bound the sum $S_{\RomanNumeralCaps{2}}$ (recall (\ref{eq:two_sums_in_separation_lemma})) as 
 		\begin{align*}
 		S_{\RomanNumeralCaps{2}} &\leq N \sum_{\al=\lceil n^{1-2\e}/2\pi \rceil}^{\lfloor \e N/2\pi n \rfloor} \bP\left(\cA_0 \cap\cA_\al \right) + N \sum_{\al=\lceil \e N/2\pi n \rceil}^{\lfloor n^{2-2\e}/2\pi \rfloor} \bP\left(\cA_0 \cap\cA_\al \right) \\ & \ll \frac{(\log n)^2}{N} \sum_{\al=\lceil n^{1-2\e}/2\pi \rceil}^{\lfloor \e N/2\pi n \rfloor} n^{8\e} + \frac{(\log n)^2}{N}  \sum_{\al=\lceil \e N/2\pi n \rceil}^{\lfloor n^{2-2\e}/2\pi \rfloor} C_\e \\ & \ll \frac{(\log n)^2 n^{8\e}}{n} + \frac{(\log n)^2}{n^\e} = o(1).
 		\end{align*}
 	\end{proof}
 	\section{Liggett's invariance and proof of Theorem
	\ref{thm:limiting_distribution_of_minimum}}
	\label{sec-liggett}
	In this section, we prove the convergence in distribution
	of $\mathcal{M}_n$ to a Poisson process of constant intensity,
      following a characterization of the latter due to Liggett. In this, we follow a method developed by Biskup and Louidor in their study of the two dimensional 
      Gaussian free field \cite{biskup_louidor}. For more background, see
      the lecture notes \cite{biskup}.

      The first step is to rewrite $P_n$ as a sum of two independent polynomials.
 	Denote by $Q=Q_n$ an independent copy of the random polynomial $P=P_n$, and consider the random polynomial (of degree $n$) given by
 	\begin{equation}
 	\label{eq:def_of_polynomial_with_perturbation}
 	\widehat{P}_n(x) := \sqrt{1-\frac{1}{n^2}}P_n(x) + \frac{1}{n} Q_n(x).
 	\end{equation}
 	Clearly, $\what{P}_n \stackrel{\mbox{\scriptsize{\rm law}}}{=} P_n$, and hence
 	\begin{equation}
 		\label{eq:point_processes_have_the_same_law}
		\what{\cM}_n \stackrel{\mbox{\scriptsize{\rm law}}}{=} \cM_n,
 	\end{equation}
 	where $\cM_n$ is the extremal process (\ref{eq:definition_of_extremal_process}) and $\what{\cM}_n$ is the extremal process that corresponds to the polynomial $\widehat{P} = \widehat{P}_n$. The goal of this section is to study the relation between these two point processes. Let $\widehat{X}_\al$ and $\what{Y}_\al$ be the analogous variables to $X_\al$ and $Y_\al$, which correspond to the polynomial $\what{P}$ instead on $P$, see (\ref{eq-defA}).
 	
 	By Corollary \ref{cor:tightness_of_point_process} (tightness of the sequence $\{\cM_n\}$), we can find a subsequence $\{n_k\}\subset \bN$ so that
 	\begin{equation}
 		\label{eq:subsequential_convergence_of_both_processes}
 		\cM_{n_k}\xrightarrow{\ d \ } \cM_\infty, \quad \what{\cM}_{n_k} \xrightarrow{\ d \ } \ \what{\cM}_\infty,
 	\end{equation}
 	as $k\to\infty$. We denote the law of the limiting point process $\cM_\infty$ by $\eta$, then by (\ref{eq:point_processes_have_the_same_law}) it is evident that
 	\[
	  \eta \stackrel{\mbox{\scriptsize{\rm law}}}{=} \what{\cM}_\infty.
 	\]
	The following lemma is the key element in
	extracting the Poisson limit, which is a consequence of $\eta$ having an invariance property. As usual, for $f\in C_c(\bR)$ we denote the linear statistics of a point process $W$ by
	\[
	\avg{W,f} := \sum_{w\in W} f(w).
	\]
 	\begin{lemma}
 		\label{lemma:limiting_pp_satisfy_invariance}
 		Let $\eta$ be given as above and let $f\in C_c(\bR)$ be a non-negative function. Then
 		\begin{equation}
 			\label{eq:liggett_property_for_pp}
 			\bE\left[e^{-\avg{\eta,f}}\right] = \bE\big[ e^{-\langle\eta,\ti{f}\rangle}\big],
 		\end{equation}
 		where
 		\[
		  \ti{f}(x):= -\log \bE_{G} \left[e^{-f(x + G)} \right]
 		\]
		and $G\sim \mathcal{N}_{\bR} (0, 1/2)$. Here $\bE_{G}$ denotes the expectation with respect to the Gaussian variable $G$.
 	\end{lemma}
 	Before proving Lemma \ref{lemma:limiting_pp_satisfy_invariance}, we will need two simple results. Lemma \ref{lemma:almost_independent_gaussians_are_like_gaussians} gives a quantitative approximation for ``almost-independent" normal variables as truly independent normal variables. Lemma \ref{lemma:points_in_both_processes_are_similar} tells us that with high probability, the perturbation (\ref{eq:def_of_polynomial_with_perturbation}) did not introduce any new points into the extermal process, nor did it delete the points that where present before the perturbation.
 	
	\begin{lemma}[{\cite[Lemma 4.9]{biskup_louidor}}]
 		\label{lemma:almost_independent_gaussians_are_like_gaussians}
 		Fix $\sigma^2>0$ and $m\ge 1$. Suppose that $G_1,\ldots, G_m$ are i.i.d. $\mathcal{N}_{\bR}(0,\sigma^2)$. For each $f\in C_c(\bR)$ and for each $\e>0$ there exist $\delta>0$ such that if $\left(W_1,\ldots,W_m\right)$ are multivariate normal with $\bE[W_i] = 0$ and 
 		\[
 		\max_{1\leq i,j\leq m} \left|\bE\left[W_i W_j\right] - \sigma^2\delta_{i,j}\right|\leq \delta 
 		\]
 		then
 		\[
 		\left|\log \frac{\bE \exp \left\{\sum_{i=1}^{m} f(X_i)\right\}}{\prod_{i=1}^{m}\bE (e^{f(Y_i)})}\right| < \e.
 		\]
 	\end{lemma}
 	
 	\begin{lemma}
 		\label{lemma:points_in_both_processes_are_similar}
 		We have that
 		\[
 		\lim_{K\to\infty} \limsup_{n\to\infty} \bP\left(\bigcup_{\al = -N/2}^{N/2-1} \left\{ |X_\al| \leq K, \ |\widehat{X}_\al| > 2K \right\} \right) = 0.
 		\]
 	\end{lemma}
 	\begin{proof}
 		We turn to bound $\bP\left( |X_\al| \leq K, \ |\widehat{X}_\al| > 2K \right)$ for each $\al$, and assume by stationarity that $x_\al = 0$. We have,
 		\begin{align*}
 			\bP&\left( |X_\al| \leq K, \ |\widehat{X}_\al| > 2K \right) \\ &\leq \bP\left( |X_\al| \leq K, \ \what{Y}_\al \not \in I_\al\right) + \bP\left(|X_\al| \leq K, \ |\widehat{X}_\al| > 2K, \ \what{Y}_\al \in I_\al \right)  =: \bP(E_1) + \bP(E_2).
 		\end{align*}
 		To bound the probabilities $\bP(E_i)$, $i=1,2$ we exploit the relation between the polynomial and its small perturbation (\ref{eq:def_of_polynomial_with_perturbation}). Denote by
 		\[
 		{\begin{array}{cc}
 		F(x) = A+ Bx  \\ \widehat{F}(x) = \widehat{A}+ \widehat{B}x
 		\end{array}}
 		 \quad \text{where, } \quad 
 		 { \begin{array}{cc}
 		 	A = P_n(0), & B = P_n^\prime(0), \\ \widehat{A} = \widehat{P}_n(0),  & \widehat{B} = \widehat{P}_n^\prime(0).
 		 	\end{array}}
 		\]
 		We see that
 		\begin{align*}
 		\widehat{A} = \sqrt{1-\frac{1}{n^2}} A + \frac{1}{n} Q(0), \quad \widehat{B} = \sqrt{1-\frac{1}{n^2}} B + \frac{1}{n} Q^\prime(0).
 		\end{align*}
 		By Taylor expanding the square root and using the fact that $|A|\leq n^{-1/2}$ and $|B|\in \left[n^{1-\e/2}, C_0 n\sqrt{\log n} \right] $ on the event $\cA_\al$,
		we see that,
 		\begin{equation}
 	\label{eq:difference_in_linear_approximation_after_perturbation}
 			\left|A - \what{A} + \frac{Q(0)}{n} \right| \ll \frac{1}{n^{3/2}}, \qquad \left|B - \what{B}\right| \ll \frac{|Q^\prime(0)|}{n} + \frac{\sqrt{\log n}}{n}.
 		\end{equation}
 		Therefore, on the event $\cA_\al \cap \left\{|Q(0)|\leq n^\e , \ |Q^\prime(0)| \le n^{1+\e} \right\}$, we have that
		$|B| = |\what{B}|(1+o(1))$ and thus 
 		\begin{align*}
 			\left|Y_\al - \what{Y}_\al\right| &= \left| \frac{\re(A\wbar{B})}{|B|^2} - \frac{\re(\what{A}\wbar{\what{B}})}{|\what{B}|^2} \right| \\ & \ll \frac{\left|\re\left(A \wbar{Q^\prime}(0) + Q(0)\wbar{B} \right)\right|}{n|B|^2} \ll \frac{1}{n^{3/2-2\e}} + \frac{|Q(0)|}{n^{3-2\e}} \ll \frac{1}{n^{3/2 - 2\e}}.
 		\end{align*}
 		Therefore,
 		\begin{align}
 			\bP\left(E_1\right) 
			&\leq \bP\left(|X_\al|  \leq K , \ |Y_\al| \in \left[\frac{\pi}{N} - \frac{1}{n^{3/2 - 2\e}}, \frac{\pi}{N}\right]  \right)  \nonumber\\
			&\qquad + \bP\left(\left\{|Q(0)|\ge n^\e \right\} \cup \left\{|Q^\prime(0)| \ge n^{1+\e}  \right\} \right) \nonumber \\ \label{eq:bound_on_prob_of_E_1} &\ll \frac{1}{N n^{3/2 - 2\e}} + e^{ - n^\e}.
 		\end{align}
 		To bound $\bP(E_2)$, we use (\ref{eq:difference_in_linear_approximation_after_perturbation}) once more and see that on the event $E_2$,
 		\begin{align*}
 			\left|X_\al - \what{X}_\al\right| &= \left|\frac{\im(A \wbar{B})}{|B|} - \frac{\im(\what{A} \wbar{\what{B}})}{|\what{B}|} \right| \ll \frac{\left|\im\left( Q(0)\wbar{B}  \right)\right|}{|B|} \ll |Q(0)|,
 		\end{align*}
 		which implies that
 		\begin{align}
 			\label{eq:bound_on_prob_of_E_2}
 			\bP\left(E_2\right)& \leq \bP\left(|X_\al| \leq K , \ |Q(0)| \geq \frac{K}{2} \right)\\ \nonumber &= \bP\left(|X_\al|\leq K\right) \cdot \bP\left(|Q(0)| \geq K/2\right) \ll \frac{Ke^{-K^2/4}}{N}.
 		\end{align}
 		Combining the bounds (\ref{eq:bound_on_prob_of_E_1}) and (\ref{eq:bound_on_prob_of_E_2}) together with the union bound, we see that
 		\begin{align*}
 			\bP\left(\bigcup_{\al = 1}^{N} \left\{ |X_\al| \leq K, \ |\widehat{X}_\al| > 2K \right\} \right) &\leq N\bP\left( |X_\al| \leq K, \ |\widehat{X}_\al| > 2K \right) \\ &\ll \frac{1}{n^{3/2-2\e}} + Ke^{-K^2/4}.
 		\end{align*}
 	\end{proof}
 	\begin{proof}[Proof of Lemma \ref{lemma:limiting_pp_satisfy_invariance}]
 		Fix $f\in C_c(\bR)$, and assume that the support of $f$ is strictly contained in $(-K,K)$ for some $K>0$ (large enough). Let $\mathcal{F}_P$ denote the $\sigma$-algebra generated by the coefficients $\left\{\zeta_j \right\}_{j=-n}^{n}$ of the polynomial $P$.
		We have that
 		\begin{align}
 			\nonumber
 			\bE\left[e^{-\avg{\eta,f}}\right] &= \lim_{k\to\infty} \bE\left[\exp\left(-\avg{\what{\cM}_{n_k},f} \right)\right] \\ &= \label{eq:conditional_expectation_for_linear_statistics} \lim_{k\to\infty} \bE\bigg[\bE\Big[\exp\Big(-\sum_{\al=1}^{N_k} f(\what{X}_\al)\Big) \mid 
			\mathcal{F}_P\Big] \bigg],
 		\end{align}
 		where here $N_k = 2\lfloor n_k^{2-\e}/2\rfloor$. From the estimates
		(\ref{eq:difference_in_linear_approximation_after_perturbation}), we see that outside of an event of $o(1)$ probability, we have that
		for all $\alpha \in [N]$,
 		\[
 		\what{X}_\al = X_\al + \frac{\re(Q(x_\al)\wbar{P^\prime(x_\al)})}{|P^\prime(x_\al)|} + \cO\left(\frac{1}{\sqrt{n}}\right),
 		\]
		and the error term is uniform in $\al$. Denote by $G_\al := \re(Q(x_\al)\wbar{P^\prime(x_\al)})/|P^\prime(x_\al)|$. Then, conditioned on $\mathcal{F}_P$, 
		the random variables $\{G_\al\}$ are jointly normal,
		and each $G_\alpha$ has law $\mathcal{N}_{\bR}(0,1/2)$. Using Lemmas \ref{lemma:points_in_the_process_are_separated} and \ref{lemma:points_in_both_processes_are_similar}, 
		we obtain that outside of an event of probability $o(1)$, we have for all $\al,\al^\prime\in\left\{ \al\in [N] :  |X_\al|\leq 2K \right\}$ that
		\begin{equation}
		  \label{eq-cov}
		  |\bE\left[G_\al G_{\al^\prime} \mid \mathcal{F}_P\right]| \ll \left|\bE\left[Q(x_\al) \wbar{Q(x_{\al^\prime})}\right]\right| = \left|r_n(x_\al-x_{\al^\prime})\right| 
		\ll \frac{1}{n^{1-\e}},
	      \end{equation}
	      where $r_n$ is again as in (\ref{eq:covariance_kernel_of_polynomial}). Putting everything together, we use Lemma \ref{lemma:almost_independent_gaussians_are_like_gaussians} and the uniform continuity of $f$ to see that
 		\begin{align*}
		  \bE\bigg[\exp\Big(-\sum_{\al=1}^{N_k} f(\what{X}_\al)\Big) \mid \mathcal{F}_P\bigg] &= e^{o(1)} \bE\bigg[\exp\Big(-\sum_{\al=1}^{N_k} f(X_\al + G_\al)\Big) \mid \mathcal{F}_P\bigg] \\ &= e^{o(1)} \exp\Big(-\sum_{\al=1}^{N_k} \ti{f}(X_\al)\Big)
 		\end{align*}
 		where the $o(1)$ term may be random (measurable on $\mathcal{F}_P$, but still, it is of order $o(1)$ with probability approaching $1$ as $n\to\infty$.)
	      Plugging into (\ref{eq:conditional_expectation_for_linear_statistics}) and using (\ref{eq:subsequential_convergence_of_both_processes}) we get that
 		\begin{equation*}
 			\bE\left[e^{-\avg{\eta,f}}\right] = \lim_{k\to\infty} \bE\Big[\exp\Big(-\sum_{\al=1}^{N} \ti{f}(X_\al)\Big)\Big] = \bE\big[ e^{-\langle\eta,\ti{f}\rangle}\big]
 		\end{equation*}
 		as desired.
 	\end{proof}
 	Finally, we extract the Poisson limit from relation (\ref{eq:liggett_property_for_pp}) and with that the proof of Theorem \ref{thm:limiting_distribution_of_minimum}.
 	\begin{proposition}
 		\label{prop:invariance_property_implies_poisson}
 		Suppose that $\eta$ is a point process on $\bR$ such that 
		(\ref{eq:liggett_property_for_pp}) holds for 
		all non-negative $f\in C_b(\bR)$. Then $\eta$ is a Poisson point process whose intensity measure $\mu$ is a constant multiple of the Lebesgue measure.
	\end{proposition}
	\begin{proof}
	  From (\ref{eq:liggett_property_for_pp}) we know that the law of $\eta$ is an invariant measure for the transformation that
	  consists of adding to each point in the support of
	  $\eta$ an independent mean-zero Gaussian variable of variance $1/2$. Therefore, by \cite[Theorem 4.11]{liggett}, the law of $\eta$ is a mixture of Poisson processes
		of intensities $\mu$ that  satisfy the relation  
		\begin{equation}
		  \label{eq-mu}
		  \mu \star \mathcal{N}_{\bR}(0,1/2) = \mu
		\end{equation}
		where $\star$ denotes the convolution of two measures. By the result of Deny \cite[Theorem~3']{deny} (based on Choquet-Deny \cite{choquet_deny}),
		we know that any solution of (\ref{eq-mu}) is of the form 
		\[
		\mu = \left(\int_{\bR} e^{-\rho x} d\nu(\rho) \right) dx
		\]
		where $\nu$ is a measure supported on those exponential functions $e_\rho(x):= e^{-\rho x}$ which satisfy $e_\rho\star \mathcal{N}(0,1/2) = e_{\rho}$. A straight forward computation shows that,
		\[
		 1 = e^{\rho x} \int_{\bR} e^{-\rho (x-y)} e^{-y^2} \frac{dy}{\sqrt{\pi}} = e^{\rho^2/4}
		\] 
		which in turn implies that $\rho =0$. Thus, we conclude that the measure $\nu$ is a constant multiple of a delta point mass at $\rho =0$. That is, $\mu$ is some multiple of the Lebesgue measure on $\bR$. Since convex combinations of Poisson processes with constant intensity yield a Poisson process of some (constant) intensity, we conclude that $\eta$ is a Poisson process with a constant intensity, which proves the proposition.
  	\end{proof}
 	\begin{proof}[Proof of Theorem \ref{thm:limiting_distribution_of_minimum}]
 		By Proposition \ref{prop:invariance_property_implies_poisson}, we know that $\{\cM_n\}$ converges on a subsequence
		to a Poisson process with intensity that is a multiple of the Lebesgue measure. By Corollary \ref{cor:limit_intensity_of_point_process}, we know that the limit of the intensity of $\cM_n$ is $\sqrt{\pi/3}$ times the Lebesgue measure. Since the limiting process does not depend on the subsequence, we use the tightness once more and conclude that $\cM_n$ converge to a Poisson point process with this given intensity. It remains to apply Lemma \ref{lemma:extremal_process_captures_minimum} and see that
 		\[
 		\lim_{n\to\infty} \bP\left(m_n \ge \frac{\tau}{n}  \right) = \lim_{n\to\infty} \bP\left(\cM_n((-\tau,\tau))\right) = \exp\left(-2\sqrt{\frac{\pi}{3}}\tau \right). 
 		\]
 	\end{proof}
 	
 	\section{Real Gaussian coefficients}
 	\label{sec:real_gaussian}
 	In this section we briefly comment on the analogous result to Theorem \ref{thm:limiting_distribution_of_minimum} in the case of real Gaussian coefficients. Let $\{X_j\}$ be an i.i.d. sequence of $\mathcal{N}_{\bR} (0,1)$ random variables and consider the random trigonometric polynomial given by
	\begin{align}
		\label{eq:definition_of_real_polynomial}
 		T_{n}(x) &:= \frac{1}{\sqrt{2n+1}} \sum_{j=-n}^{n} X_j e^{ijx} \\ &= \frac{1}{\sqrt{2n+1}} \left(\sum_{j=-n}^{n} X_j \cos(jx) + i\sum_{j=-n}^{n} X_j \sin(jx) \right) =: R_n(x) + i I_n(x). \nonumber
 	\end{align}
 	As before, we denote by $m_n(T) = \min_{x\in \bT} |T_n(x)|$.
 	\begin{theorem}
 		\label{thm:real_gaussian_case}
 		For any $\tau>0$ we have that
 		\[
 		\lim_{n\to\infty} \bP\left(m_n(T) \ge \frac{\tau}{n}\right) = e^{-\lambda \tau}
 		\]
 		where $\lambda= 2\sqrt{\pi/3}$.
 	\end{theorem} 
 	The proof the Theorem \ref{thm:real_gaussian_case} is almost identical to that of Theorem \ref{thm:limiting_distribution_of_minimum}, only the computations are more cumbersome. The reason for this complication is that the polynomial $T_n$ is no longer stationary (as opposed to $P_n$). Still, $T_n$ is a complex-valued Gaussian process on $\bT$ so computations are possible.
 	
 	The correlations of the real and imaginary part of $T=T_n$ are given by
 	\begin{align}
 		\label{eq:correlations_in_the_real_case}
 		&\bE\left[R(x) R(y)\right] = \frac{1}{2n+1} \sum_{j=-n}^{n} \cos(jx)\cos(jy) = \frac{r_n(x-y) + r_n(x+y)}{2} \nonumber \\ & \bE\left[I(x) I(y)\right] = \frac{1}{2n+1} \sum_{j=-n}^{n} \sin(jx)\sin(jy) = \frac{r_n(x-y) - r_n(x+y)}{2} \\ & \bE\left[R(x) I(y)\right] = \frac{1}{2n+1} \sum_{j=-n}^{n} \cos(jx)\sin(jy) = 0, \nonumber
 	\end{align}
 	where $r_n$ is given by (\ref{eq:covariance_kernel_of_polynomial}) and $x,y\in \bT$. Recall the definition of the event that the interval $I_\al$ produced a candidate for a minimal value (\ref{eq-defA}). It follows from (\ref{eq:correlations_in_the_real_case}) that as long as $|x_\al|\in [n^{-1+\e}, \pi - n^{-1+\e}] $ then
 	\[
 	\bE\left[R(x_\al)^2\right] = \frac{1}{2} + o(1) \quad \text{and } \quad \bE\left[I(x_\al)^2\right] = \frac{1}{2} + o(1).
 	\]
 	That is, the random variable $T(x_\al)$ scale as a standard complex Gaussian and similar computations as in Lemma \ref{lemma:probability_of_single_point_to_be_near_minima} can be carried out with no problems. Still, we need to show that with probability tending to 1 the minimum of $T$ does not occur inside $\left\{|x|\leq n^{-1+\e}\right\} \cup \left\{|x-\pi|\leq n^{-1+\e}\right\} $, this we do in Lemma \ref{lemma:minimum_does_not_occur_inside_small_interval_real_coefficients} below.
 	
 	In proving that the points of the extremal process are obtained from well separated intervals (i.e. the analogous result to Lemma \ref{lemma:points_in_the_process_are_separated}), we remark that proving Claim \ref{claim:first_sum_in_seperatin_lemma_is_small} for the real coefficients case is straight forward. To prove Claim \ref{claim:second_sum_in_seperatin_lemma_is_small} for the real case, one can use (\ref{eq:correlations_in_the_real_case}) while noticing that
 	\begin{equation*}
 		r_n(x+y) \ll r_n(2x) + |x-y|r_n^\prime (2x) \ll \frac{1}{n|x|} + |x-y| = o\left( (n|x-y|)^2\right)
 	\end{equation*} 
 	provided that $|x|\ge n^{-1+\eps}$ and that $n^{-1-2\eps}\leq |x-y|\le n^{-\eps}$. The proof of Liggett's invariance (Section \ref{sec-liggett}) also translates to the case of real coefficients with no problems.
 	
 	It remains to prove that with high probability the intervals $$\left\{|x|\leq n^{-1+\e}\right\} \cup \left\{|x-\pi|\leq n^{-1+\e}\right\}$$ will not contribute a point to the extermal process that corresponds to $T$. Since $T(x)$ and $T(x+\pi)$ have the same distribution, it suffices to consider the interval centered around $0$.
 	\begin{lemma}
 		\label{lemma:minimum_does_not_occur_inside_small_interval_real_coefficients}
 		Let $T = T_n$ be given as (\ref{eq:definition_of_real_polynomial}), then
 		\[
 		\lim_{n\to\infty }\bP\left(\min_{|x|\leq n^{-1+\e}}|T(x)| \leq \frac{\log n}{n}\right) = 0.
 		\] 
 	\end{lemma}
 	\begin{proof}
 		By repeating the exact same argument as in Lemma \ref{lemma:second_derivative_is_small}, we can show that for any $\e>0$
 		\[
 		\bP\left(\norm{T^\prime}_\infty \ge n^{1+\e} \right) \ll \exp\left(-n^{\e}\right).
 		\]
 		We cover the interval $\left\{n^{-1-2\e} \leq |x|\leq n^{-1+\e} \right\}$ by non-overlapping intervals $\left\{J_\ell\right\}_{\ell = 1}^{L}$ so that $|J_\ell| \leq n^{-5/4}$, and let $y_\ell\in J_\ell$ be some point. Notice that $L\leq n^{1/4+\e}$. On the event $\left\{\norm{T^\prime}_\infty \le n^{1+\e}\right\}$, we know by Lagrange's mean value theorem that
 		\[
 		\left|T(x) - T(y_\ell)\right| \leq n^{-5/4} \norm{T^\prime}_{\infty} \leq n^{-1/4+\e}, \quad \text{for all } x\in J_\ell.
 		\]
 		Thus, on the event $\left\{\min_{x\in J_\ell} |T(x)| \leq \frac{\log n}{n} \right\} \cap \left\{ \norm{T^\prime}_\infty \le n^{1+\e} \right\}$, we know that
 		\begin{equation}
 			\label{eq:real_case_bound_on_point_if_minimum_occur}
 			|T(y_\ell)| \leq \frac{\log n}{n} + n^{-5/4} n^{1+\e} \le 2n^{-1/4+\e}
 		\end{equation}
 		for $n$ large enough. Recall from (\ref{eq:correlations_in_the_real_case}) that $T(y_\ell) = (R(y_\ell) ,I(y_\ell))$ is a bivariate mean-zero Gaussian with correlations given by
 		\[
 		\begin{bmatrix}
 		\frac{1}{2n+1} \sum_{j=-n}^{n} \cos^2(jy_\ell) & 0 \\ 0 & \frac{1}{2n+1} \sum_{j=-n}^{n} \sin^2(jy_\ell) 
 		\end{bmatrix} =: \begin{bmatrix}
 		1-\Lambda_\ell & 0 \\ 0 & \Lambda_\ell 
 		\end{bmatrix}.
 		\]  		
 		For $|n y_\ell| \leq \pi/2$ we have that
 		\begin{align}
		  \label{eq-Lam}
 			\Lambda_\ell \geq \frac{4y_\ell^2}{\pi^2(2n+1)} \sum_{j=-n}^{n} j^2 \gg n^{-2-4\e} n^2 \gg n^{-4\e}. 
 		\end{align}
		Moreover, for $|ny_\ell| \leq \pi/2$, the density of $T(y_\ell)$ is uniformly bounded from above. Using (\ref{eq:real_case_bound_on_point_if_minimum_occur}) and (\ref{eq-Lam}), we see that
 		\begin{align*}
 			\bP\left(\min_{|x|\leq n^{-1+\e}}|T(x)| \leq \frac{\log n}{n}\right) & \leq \bP\left(\min_{|x|\leq n^{-1-2\e}} |T(x)| \leq \frac{\log n}{n},\  \norm{T^\prime}_\infty \le n^{1+\e}\ \right) \\ &\quad + \sum_{\ell =1}^{L} \bP\left(|T(y_\ell)| \leq 2 n^{-1/4+\e}\right) + \bP\left(\norm{T^\prime}_\infty \ge n^{1+\e} \right) \\ & \ll \bP\left(|T(0)|\leq 2n^{-\e}\right) + n^{-1/4+7\e} +e^{-n^\e} = o(1).
 		\end{align*}
 	\end{proof}
	
	\section{Coefficients with small Gaussian component}
	\label{sec-small_gaussian_component}
	In this section we briefly sketch an extension of Theorem \ref{thm:real_gaussian_case} to a more general choice of coefficients. Recall that a
	real-valued 
	random variable $\xi$ satisfies \emph{Cram\'{e}r's condition} if
	\begin{equation}
		\label{eq:cramer_condition}
		\limsup_{t\to\infty} \left|\bE\left[e^{it\xi}\right]\right| < 1.
	\end{equation}
	In particular, (\ref{eq:cramer_condition}) is satisfied when the 
	distribution of $\xi$ possesses 
	an absolutely continuous density, 
	but there are also many singular distributions
	which satisfy (\ref{eq:cramer_condition}). Now, 
	let $\delta_n\in (0,1)$ be any sequence such that 
	$\delta_n \gg (\log n) /n$ and consider the random polynomial
	(analogous to (\ref{eq:definition_of_real_polynomial}))
	\begin{equation}
	  \label{eq:perturbed}
		T_n(x) = \frac{1}{\sqrt{2n+1}} \sum_{j=-n}^{n} \left(\xi_j + \delta_n X_j\right) e^{ijx}
	\end{equation}
	where $\xi_j$ are i.i.d. random variables whose
	common distribution satisfies (\ref{eq:cramer_condition})
	and where $X_j$ are i.i.d. $\mathcal{N}_{\bR} (0,1)$ 
	and independent of the $\xi_j$'s. We will 
	further assume the normalization conditions 
	$\bE[\xi] = 0$ and $\bE[\xi^2] = 1-\delta_n^2$.

To show that Theorem \ref{thm:real_gaussian_case} remains true with 
(\ref{eq:perturbed}) replacing (\ref{eq:definition_of_real_polynomial}),
we follow the same steps as in the proof of 
Theorem \ref{thm:real_gaussian_case}. First,
	we need to prove results analogous to Section 
	\ref{sec-point_process_of_near_minima} 
	for the process of near-minima values. This amounts to proving 
	a local limit theorem for the vector 
	$\left(T_n(x_\alpha),T_n^\prime(x_\alpha)\right)\in \bR^4$ 
	for $|x_\alpha|\in [n^{-1+\e}, \pi - n^{-1+\e}]$ 
	(and, for Claim \ref{claim:second_sum_in_seperatin_lemma_is_small}, 
	the vector $\left(T_n(x_{\alpha}),T_n^\prime(x_{\alpha}),T_n(x_{\beta}),T_n^\prime(x_{\beta})\right)\in \bR^8$ for 
	$\alpha\not=\beta$). 
	As the correlations of $T_n$ scale exactly like 
	(\ref{eq:correlations_in_the_real_case}), this local limit result 
	can be extracted from an Edgeworth expansion of order three, 
	see for example \cite[Corollary~20.4]{bha_rao}. 
	The condition (\ref{eq:cramer_condition}) allows us to apply 
	the corresponding Edgeworth expansion. 
	To deal with the intervals $\left\{|x|\leq n^{-1+\e}\right\}$ and 
	$\left\{|x-\pi|\leq n^{-1+\e}\right\}$ 
	one obtains the analogue of Lemma 
	\ref{lemma:minimum_does_not_occur_inside_small_interval_real_coefficients} by using the basic Berry-Esseen inequality, see
	\cite[Lemma~3.3]{konyagin_schlag}. 
	
	It remains to explain how one can prove 
	the results from Section \ref{sec-liggett} for $T_n$ of
	(\ref{eq:perturbed}). 
	Write $$X_j = X_j^{\prime} + \frac{1}{\delta_n n}X_j^{\prime\prime}$$
	where $X_j^\prime,X_j^{\prime\prime}$ are independent mean-zero Gaussians and $\text{Var}(X_j^{\prime\prime}) = 1$. Now, the polynomial $T_n$ splits into three polynomials
	\begin{equation*}
		T_n(x) = \ti{T}_n(x) + R_n(x) + \frac{1}{n} G_n(x),
	\end{equation*}
	where $T_n \stackrel{\mbox{\scriptsize{\rm law}}}{=} \ti{T}_n$, $G_n$ is a Gaussian polynomial given as in (\ref{eq:definition_of_real_polynomial}) independent of $\ti{T}_n$ and
	\begin{equation*}
		R_n(x) = \frac{\delta_n\left(\sqrt{1-1/(\delta_n n)^2}-1\right)}{\sqrt{2n+1}}\sum_{j=-n}^{n} X_j^{\prime} e^{ijx}.
	\end{equation*}
	(Note that $R_n$ is \textit{not} independent of $\ti{T}_n$.) 
	By virtue of the 
	Salem-Zygmund inequality \cite[Chapter~6, Theorem~1]{kahane},
	we get that 
	\begin{align*}
		n \cdot \sup_{x\in \bT} |R(x)| &\ll n \sqrt{\log n} \cdot \delta_n\left(\sqrt{1-1/(\delta_n n)^2}-1\right) 
		\ll \frac{\sqrt{\log n}}{n \delta_n} = o(1), 
	\end{align*}
	with probability tending to 1 as $n\to\infty$. 
	Hence, the polynomial $R_n(x)$ does not affect the 
	limiting process of near minima values. Clearly, the independent ``kicks" given by the polynomial $n^{-1}G_n(x)$ correspond to the i.i.d. perturbations of the limiting process and we arrive at the setting of Liggett's theorem. 
	
	\subsection*{Acknowledgments} O. Y.  thanks Alon Nishry for introducing him to this problem, encouraging him to work on it and for many fruitful discussions.
	O. Z. thanks Hoi Nguyen for suggesting this problem 
	at an AIM meeting in August 2019, and 
	Pavel Bleher, Nick Cook and Hoi Nguyen for stimulating 
	discussions at that meeting concerning this problem. 
	In particular, the realization that a linear approximation suffices
	when working on a net with spacing  $o(1/n)$ came out of those discussions. 
	
 	\end{document}